\documentclass[11pt]{amsart}

\usepackage[english]{babel} 
\usepackage[T1]{fontenc}
\usepackage[utf8]{inputenc}

\usepackage{amsmath,amssymb,mathrsfs} 
\usepackage{graphicx}  
\usepackage{float, caption} 
\usepackage{wrapfig} 
\usepackage{todonotes}

\usepackage{framed}
\usepackage{array, multirow, tabularx} 
\usepackage{tikz} 
\usepackage{pifont} 
\usetikzlibrary{matrix,arrows,decorations.pathmorphing,patterns}
\usepackage{mathtools}
\usepackage{enumerate}

\usepackage[bookmarksnumbered=true,pdfstartview=FitH,pdftex,pdfborder={0 0 0},linkcolor=blue,citecolor=blue,colorlinks=true]{hyperref} 

{\newtheorem{de}{Definition}[section]}

{
\newtheorem{theo}[de]{\textsc{Theorem}}}
{
}

{
\newtheorem{prop}[de]{Proposition}}

{
\newtheorem{lem}[de]{Lemma}}

{
\newtheorem{cor}[de]{Corollary}}

{\theoremstyle{remark}
\newtheorem{remar}[de]{Remark}}

{\theoremstyle{remark}
}

{\theoremstyle{remark}
}

{\theoremstyle{remark}
\newtheorem{ex}[de]{Example}}

{\theoremstyle{remark}
\newtheorem{nota}[de]{Notation}}

{
}

{

{}
{
}

\newcommand{\CC}{\mathbb{C}}
\newcommand{\C}{\mathbb{C}}

\newcommand{\NN}{\mathbb{N}}

\newcommand{\Sing}{\mathrm{Sing}}
\newcommand{\depth}{\mathrm{depth}}

\newcommand{\wt}[1]{\ensuremath{\widetilde{#1}}}
\newcommand{\und}[1]{\underline{#1}}

\newcommand{\m}{\ensuremath{\mathfrak{m}}}

\newcommand{\Der}{\mathrm{Der}_{\CC^n}}
\newcommand{\projdim}[1]{\mathrm{projdim\lrp{#1}}}
\newcommand{\dr}{\partial}
\newcommand{\dd}{\mathrm{d}}

\newcommand{\lra}[1]{\left\{#1\right\}}
\newcommand{\lrp}[1]{\left(#1\right)}
\newcommand{\lrb}[1]{\left\langle#1\right\rangle}

\newcommand{\mc}[1]{\mathcal{#1}}

\title{On constructions of free singularities}

\author[R.~Epure]{Raul Epure}
\address{ \linebreak
Raul Epure\\
Department of Mathematics, TU Kaiserslautern\\
  67663 Kaiserslautern\\
  Germany
 }
 \email{\href{epure@mathematik.uni-kl.de}{epure@mathematik.uni-kl.de}}

\author[D.~Pol]{Delphine Pol}
\address{ \linebreak
Delphine Pol\\
Department of Mathematics, TU Kaiserslautern\\
  67663 Kaiserslautern\\
  Germany
 }
 \email{\href{pol@mathematik.uni-kl.de}{pol@mathematik.uni-kl.de}}
 \date{\today}
\thanks{DP is supported by a Humboldt Research Fellowship for Postdoctoral Researchers.}


\usepackage{verbatim}

\newcommand{\infe}{<}

\newcommand{\Derlog}[2]{\mathrm{Der}^{#2}\left(-\log #1\right)}
\newcommand{\Jac}{\mathrm{Jac}}

\begin{document}

\begin{abstract}
The purpose of this paper is to give new examples of families of free singularities. We first show that a generic equidimensional subspace arrangement is free. Furthermore, we show that a product of two reduced Cohen-Macaulay subspaces is free if and only if both subspaces are free.
\end{abstract}

\maketitle

\section{Introduction}

The study of free divisors was initiated with the work of K. Saito in \cite{saito75} and \cite{Sai80}, and developed in the case of hyperplane arrangements in \cite{orlik-terao-hyperplanes}. Known families of free divisors are for example the discriminant of a deformation of an isolated hypersurface singularity (see \cite{Sai80}) or reflection arrangements (\cite{orlik-terao-hyperplanes}). 

A generalization of the notion of free divisors to complete intersections  is suggested in  \cite{gsci}, which is then extended to Cohen-Macaulay subspaces and equidimensional subspaces in \cite{polthese} and \cite{polfreeci}. Basic examples of free singularities are given in \cite{polfreeci}: curves and arbitrary unions of equidimensional coordinate subspaces. 

\smallskip

The purpose of this paper is to give new families of free singularities. 

\smallskip

We first show that a generic equidimensional subspace arrangement of codimension $k$ in $\C^n$ is free if the number of subspaces is lower than or equal to $\binom{n}{k}$ (see Theorem~\ref{theo:gen:free}). 

We notice that the singular locus of a Thom-Sebastiani sum of non-smooth normal crossing divisors is free, whereas the divisor itself is not free (see Lemma~\ref{lem:normal:cros}).
Since the singular locus of the aforementioned divisor is the product of the singular locus of the two individual divisors, the question of investigating the relation between freeness and products arises. We show that a product of two reduced Cohen-Macaulay subspaces is free if and only if the two subspaces are free (see Theorem~\ref{theo:prod}). In the particular case of divisors, it follows that the product of two divisors is a free complete intersection of codimension $2$ if and only if both divisors are free. 

All computations have been performed using the computer algebra system \textsc{SINGULAR}  (\cite{DGPS}). In order to compute all mentioned algebraic objects we provide the \textsc{SINGULAR}-library \texttt{logmodules.lib} which can be downloaded under \newline \emph{https://github.com/delphinepol/Free-singularities/blob/main/logmodules.lib}. 
\vspace*{.2cm}

\noindent\textbf{Acknowledgement}\textit{.} We thank Michel Granger, Mathias Schulze and the anonymous referee for helpful suggestions and comments. This paper is part of the first authors Ph.D. thesis (see \cite{muhthese}). 

\section{Preliminaries}

Let $n\in\NN_{\geqslant 1}$. Throughout this paper, if not stated otherwise, let $S$ be either $\CC[x_1,\ldots,x_n]$ or $\CC\lra{x_1,\ldots,x_n}$. For the sake of simplicity, we will also write $\C^n$ in the local case instead of $(\C^n,0)$.

We denote by $\Der$ the $S$-module of vector fields on $\CC^n$, which is a free $S$-module of rank $n$, generated by the vector fields $\lra{\dr_{x_1},\ldots,\dr_{x_n}}$.

For $q\in\NN$ we denote by $\Omega^q_{\C^n}$ the module of differential forms of degree $q$ on~$\CC^n$ and we consider the usual pairing $\lrb{\cdot,\cdot} :\bigwedge^q \Der \times \Omega^q_{\C^n} \to S$.

A generalization of the module of logarithmic vector fields along singular hypersurfaces (see \cite{Sai80}) is introduced in \cite{gsci} for complete intersections and in \cite{polfreeci} for general equidimensional subspaces. We give here the equivalent definition as stated in \cite{STfree}:

\begin{de}[\protect{\cite[Definition 3.19]{STfree}}]
\label{de:Derlog}
Let $X$ be an equidimensional subspace of codimension $k$ defined as the vanishing set of the radical ideal $I_X$. The module of multi-logarithmic $k$-vector fields along $X$ is defined by
\[
\Derlog{X}{k}=\lra{\delta\in\bigwedge^k \Der \mid \forall (f_1,\ldots,f_k)\in I_{X}, \lrb{\delta,\dd f_1\wedge \dots\wedge \dd f_k}\in I_{X}}.
\]
\end{de}

\begin{remar}
\label{remar:gen:I}
Let $\lra{h_1,\ldots,h_r}$ be a generating set of $I_{X}$.  
Let $\delta\in\bigwedge^k \Der$. 
Then $\delta\in\Derlog{X}{k}$ if and only if for all $\lrp{i_1\infe\dots\infe i_k}\subseteq \lra{1,\ldots,r}$, $\lrb{\delta,\dd h_{i_1}\wedge \dots\wedge \dd h_{i_k}}\in I_{X}$. 
\end{remar}

A reduced hypersurface $D$ is called free if and only if $\mathrm{Der}(-\log D):=\Derlog{D}{1}$ is a free $S$-module (see \cite{Sai80}). A generalization of this notion to higher codimensional subspaces is the following:

\begin{de}[\protect{\cite[Definition 4.3]{polfreeci}}]
\label{de:free}
An equidimensional reduced subspace $X\subseteq \C^n$ of codimension $k$ is called free if and only if 
\[
\projdim{\Derlog{X}{k}}=k-1.
\]
\end{de}

In the case of hypersurfaces, the criterion of Terao and Aleksandrov (\cite{terao-hyperplanes-freeness-80}, \cite{alek}) gives a characterization of freeness in terms of a property of the singular locus. It is shown in \cite{polfreeci} that this property can be extended to Cohen-Macaulay spaces.

\medskip

Let $X\subseteq \C^n$ be a reduced equidimensional subspace. One can prove that there exists a regular sequence $(f_1,\ldots,f_k)\subseteq I_X$ such that the ideal $I_C$ generated by $f_1,\ldots,f_k$ is radical (see \cite[Remark 4.3]{alektsikh} or \cite[Proposition 4.2.1]{polthese} for a detailed proof of this result). We fix such a sequence $(f_1,\ldots,f_k)$ and denote by $C$ the complete intersection defined by the ideal $I_C=\lrb{f_1,\ldots,f_k}$.

\begin{nota}[\protect{\cite[Notation 3.6]{polfreeci}}]
\label{nota:JXC}
Let $X$ be a reduced equidimensional subspace of codimension $k$ in $\C^n$ and $C$ be a reduced complete intersection of codimension $k$ in $\C^n$ containing $X$. Let $J_{X/C}=J_C+I_X$, where $J_C$ is the Jacobian ideal of $C$, that is to say, the ideal of $S$ generated by the $k\times k$ minors of the Jacobian matrix of $(f_1,\ldots,f_k)$. 
\end{nota}

\begin{remar}
The vanishing set of the ideal $J_{X/C}$ is the restriction of the singular locus of $C$ to $X$. If $X$ is not a complete intersection, it does not describe the singular locus of $X$.
\end{remar}

The following proposition generalizes \cite[Definition 5.1]{gsci}:

\begin{prop}{\cite[Proposition 4.2]{polfreeci}}
\label{prop:free:sing:loc}
Let $X\subseteq \C^n$ be a reduced equidimensional subspace of codimension $k$ in $\C^n$ and $C$ be a reduced complete intersection of codimension $k$ containing $X$. Then $X$ is free if and only if $S/J_{X/C}=0$ or $S/J_{X/C}$ is Cohen-Macaulay of dimension $n-k-1$.
\end{prop}

\begin{remar}
If $C'$ is another reduced complete intersection of codimension $k$ containing $X$, the modules $S/J_{X/C}$ and $S/J_{X/C'}$ are isomorphic as $S/I_X$-modules (see \cite[Remark 3.8]{polfreeci}).
\end{remar}

The module of multi-logarithmic $k$-vector fields of a union of reduced equidimensional subspaces of the same codimension satisfies the following property:

\begin{prop}[\protect{\cite[Proposition 5.1]{polfreeci}}]
\label{prop:inter}
Let $X$ be a reduced equidimensional subspace of codimension $k$, with irreducible components $X_1,\ldots, X_s$.
 Then:
$$\Derlog{X}{k}=\bigcap_{i=1}^s \Derlog{X_i}{k}.$$
\end{prop}

Before giving some basic motivating examples of free singularities, let us introduce the following notation:

\begin{nota}
\label{nota:kos}
We denote by $K(\underline{f})$ the Koszul complex of a sequence $(f_1,\ldots,f_k)$ in $S$: 
\begin{equation}
\label{suite:kos}
K(\und{f}) \ :\  0\to\bigwedge^k S^k\xrightarrow{d_k}\cdots\xrightarrow{d_{2}}\bigwedge^1 S^k\xrightarrow{d_1} S\to 0.
\end{equation} 

The maps $d_p$ are given by 
\[
d_p(e_{i_1}\wedge \dots \wedge e_{i_p})=\sum_{j=1}^p (-1)^{j+1} f_j e_{i_1}\wedge\dots\wedge \widehat{e_{i_j}}\wedge \dots\wedge e_{i_p}.
\]

We also set $\wt{K}(\und{f})$ the complex obtained from $K(\und{f})$ by removing the last~$S$. 

\smallskip

 The complex $0\to S\to 0$ is denoted by $\mc{C}$. 
\end{nota}

\begin{ex}
\label{ex:coord}
Let $E_0=\lra{i_1<\dots<i_k}\subseteq \lra{1,\ldots,n}$ and let $X$ be the vector subspace of $\CC^n$ defined by the regular sequence $(x_{i_1},\ldots,x_{i_k})$. Then a generating set of $\Derlog{X}{k}$ is 
\[
\lra{x_{j}\wedge_{i\in E_0}\dr_{x_{i}}\mid j\in E_0}\cup \lra{\wedge_{i\in E} \dr_{x_i}\mid E\neq E_0}.
\]

A minimal free resolution of $\Derlog{X}{k}$ is then given by 
\[
\wt{K}\lrp{(x_i)_{i\in E_0}}\oplus\bigoplus_{1\leqslant i\leqslant\binom{n}{k}-1} \mc{C}.
\]
In particular, $\projdim{\Derlog{X}{k}}=k-1$ so that $X$ is free. 
\end{ex}

More generally, the following holds:

\begin{prop}[\protect{\cite[Corollary 5.5]{polfreeci}}]
\label{prop:normal:free}
Let $X$ be an equidimensional union of coordinate subspaces. Then $X$ is free.
\end{prop}

Motivations for Section~\ref{construction} are given by the following lemmas:

\begin{lem}
\label{lem:sum:non:free}
Let $(X,0)$ be defined by $f\in\m^2\CC\lra{x_1,\ldots,x_n}$ and $(Y,0)$ be defined by $g\in\m^2\CC\lra{y_1,\ldots,y_m}$. Furthermore, assume that $f$ and $g$ are quasi-homogeneous and reduced. Then $h=f+g$ is free if and only if $f=0$ and $g$ is free or vice-versa. 
\end{lem}
\begin{proof}
Assume that both $f$ and $g$ are non-zero.

The singular locus of $h$ satisfies $(\Sing(V(h)),0)=(\Sing(X),0)\times (\Sing(Y),0).$

Thus $\dim (\Sing(V(h)),0)\leqslant n+m-4$ and by Proposition~\ref{prop:free:sing:loc}, $h$ is not free.  
\end{proof}

\begin{lem}
\label{lem:normal:cros}
Let $f\in\CC\lra{x_1,\ldots,x_n}$ and $g\in \CC\lra{y_1,\ldots,y_m}$ be the equations of non-smooth normal crossing divisors. Let $(X,0)=(V(f+g),0)$. Then $(X,0)$ is not free, whereas $(\Sing(X),0)$ is free. 
\end{lem}
\begin{proof}
The lemma follows from Lemma~\ref{lem:sum:non:free} and Proposition~\ref{prop:normal:free}. 
\end{proof}

\begin{remar}
These lemmas show that a direct sum of normal crossing divisors is not a free divisor, whereas the corresponding singular locus, which is built as a product of the individual singular loci, is a free singularity of codimension $4$. The question of the behaviour of freeness with products then naturally arises. 
\end{remar}
\begin{remar}

The motivation to consider Lemma~\ref{lem:normal:cros} arises from the following: in this setup, using \cite[Theorem 4]{HM86}, the isomorphy class of the singular locus determines the isomorphy class of the divisor, but the property of being free does not transfer from the singular locus to the divisor.
\end{remar}

\section{Generic subspace arrangements and freeness}

In this section we assume $S=\C[x_1,\ldots,x_n]$.

\begin{de}
\label{de:subsp:arr}
An equidimensional subspace arrangement of codimension $k$ in $\CC^n$ is a finite union of pairwise distinct vector subspaces of codimension $k$ in $\CC^n$. We denote by $I_{X}\subseteq S$ the radical ideal of vanishing polynomials on $X$. 
\end{de}

\begin{remar}
	The term subspace arrangement always refers to a union of \textit{vector} subspaces in contrast to the previous part, where we allowed the union of any kind of analytic subspaces.
\end{remar}

\begin{de}
Let $\delta\in\bigwedge^k\Der$. We say that $\delta$ is homogeneous of degree $p$ if there exist homogeneous polynomials $(a_{E})_{|E|=k, E\subseteq \lra{1,\ldots,n}}$ of degree $p$  such that 
\[
\delta=\sum_{\substack{E\subseteq \lra{1,\ldots,n}\\|E|=k}} \lrp{a_E \bigwedge_{i\in E} \dr_{x_i}}.
\]
\end{de}

\begin{nota}
Let $M$ be a graded $S$-module. For $p\in\NN$ we denote by $M_p$ the submodule of $M$ composed of the homogeneous elements of $M$ of degree $p$. 
\end{nota}

\begin{de}
\label{de:generic}
Let $\Lambda$ be a finite index set and let $X=\bigcup_{i\in\Lambda} X_i$ be an equidimensional subspace arrangement of codimension $k$. We say that $X$ is generic if for $j=\min\lra{|\Lambda|,\binom{n}{k}}$ and for all $I\subseteq \Lambda$ with $|I|=j$, it holds that 
\[
\dim_\C\lrp{\bigcap_{i\in I} \Derlog{X_{i}}{k}_0}=\binom{n}{k}-j.
\]
\end{de}

\begin{remar}
The condition given in Definition~\ref{de:generic} generalizes the usual definition of generic hyperplane arrangement (see \cite[Definition 5.22]{orlik-terao-hyperplanes}), since for a hyperplane $H$, $\Derlog{H}{1}_0$ is equal to the vector fields tangent to the hyperplane.
\end{remar}

\begin{remar}
If the coefficients of the defining linear equations of the irreducible components are chosen randomly, the condition of Definition~\ref{de:generic} is satisfied. This remark can be used to create examples in a computer algebra system such as \textsc{Singular} (\cite{DGPS}). 
\end{remar}

Up to a change of coordinates, it is easy to see that a generic hyperplane arrangement in $\C^n$ with at most $n$ hyperplanes is isomorphic to a normal crossing divisor, and thus is free. The purpose of this section is to prove the following generalization of this result:

\begin{theo}
\label{theo:gen:free}
Let $X=X_1\cup\ldots\cup X_s$ be an equidimensional subspace arrangement of codimension $k$ in $\CC^n$ such that for all $i\in\lra{1,\ldots,s}$, $X_i$ is a vector subspace defined by the regular sequence $(h_{i,1},\ldots,h_{i,k})$. 

If $s\leqslant\binom{n}{k}$ and $X$ is a generic subspace arrangement, then there exists a basis $\lrp{\delta_1,\ldots,\delta_{\binom{n}{k}}}$ of $\bigwedge^k\Der$ such that a minimal generating set of $\Derlog{X}{k}$ is given by 
\begin{equation}
\label{eq:theo}
\lra{h_{i,j}\delta_i \mid i\in\lra{1,\ldots,s}, j\in\lra{1,\ldots,k}} \cup\lra{\delta_i \mid i\geqslant s+1}.
\end{equation}
\end{theo}

\begin{cor}
\label{cor:free}
Let $X=X_1\cup\ldots \cup X_s$ be an equidimensional subspace arrangement of codimension $k$ in $\CC^n$ satisfying the hypothesis of Theorem~\ref{theo:gen:free}. Then $X$ is free.
\end{cor}

In order to prove Theorem~\ref{theo:gen:free}, we need the following auxiliary lemmas.

\begin{lem}
	\label{lem:one:space}
	Let $h_1,\ldots,h_k$ be $k$ linear forms defining a vector subspace $X$ of codimension $k$.
	Then for any $\delta\in\lrp{\bigwedge^k\Der}_0\setminus \Derlog{X}{k}_0$ and $\mc{B}$ a basis of $\Derlog{X}{k}_0$ a minimal generating set of $\Derlog{X}{k}$ is of the form: 
	\[\mc{B}\cup\lra{h_i \delta\mid i\in\lra{1,\ldots,k}}.\] 
\end{lem}

\begin{proof} Let $N=\binom{n}{k}.$
 By definition the following holds: \begin{equation}\label{eq:der0}
		\Derlog{X}{k}_0=\lra{\eta \in \lrp{\bigwedge^k\Der}_0 \mid \langle \eta, \dd h_{1}\wedge \dots\wedge \dd h_{k}\rangle =0}.
	\end{equation}
	Since the $h_i$ are linear forms, Equation~\eqref{eq:der0} is equivalent to saying that $\Derlog{X}{k}_0$ can be considered as a hyperplane in $\lrp{\bigwedge^k\Der}_0 \simeq \C^{N},$ hence $\dim_\C \Derlog{X}{k}_0= N-1.$ Denote by $\mathcal B$ a basis of $\Derlog{X}{k}_0.$  There exists a $ \delta \in \lrp{\bigwedge^k\Der}_0\setminus \Derlog{X}{k}_0,$ such that $\langle \delta, \dd h_{1}\wedge \dots\wedge \dd h_{k}\rangle=:\lambda \in \C\setminus\{0\}.$ Let $\nu \in \Derlog{X}{k}$ be arbitrary. Then, by the previous considerations, we can write \[\nu = a \delta +\sum_{\eta \in \mathcal{B}}b_\eta \eta,\] where $a, b_\eta \in S.$
	We obtain\[
	\langle \nu, \dd h_{1}\wedge \dots\wedge \dd h_{k}\rangle= a\langle \delta, \dd h_{1}\wedge \dots\wedge \dd h_{k}\rangle + \sum_{\eta \in \mathcal{B}}b_\eta \langle \eta, \dd h_{1}\wedge \dots\wedge \dd h_{k}\rangle=\lambda \cdot  a  \in I_X.
	\]
	This implies $a \in I_X$, hence $\Derlog{X}{k}$ is minimally generated by \[\mc{B}\cup\lra{h_i \delta\mid i\in\lra{1,\ldots,k}}.\]
\end{proof}

\begin{nota}
	Let $h=(h_1,\ldots,h_k)\in S^k$. We denote by $\Jac(h)$ the Jacobian matrix of $h$.
\end{nota}
Using an explicit coordinate change, one can refine Lemma \ref{lem:one:space} as follows:
\begin{remar}
	Let $h_1,\ldots,h_k$ be $k$ linear forms defining a vector subspace $X$ of codimension $k$. Let $\lra{i_1<\ldots< i_k}\subseteq \lra{1,\ldots,n}$. We assume that the $k\times k$ minor of $\Jac(h)$ relative to the columns indexed by $i_1,\ldots,i_k$ is non-zero. Then a minimal generating set of $\Derlog{X}{k}$ is of the form: 
	
	\begin{equation}
		\label{eq:gen:one:space}
		\lra{h_i \dr_{x_{i_1}}\wedge\dots\wedge \dr_{x_{i_k}} \mid i\in\lra{1,\ldots,k}}\cup \lra{ \delta_2,\ \ldots,\ \delta_{\binom{n}{k}-1}},
	\end{equation}
	where for $i\in\lra{2,\ldots,\binom{n}{k}-1}$, $\delta_i$ is homogeneous of degree $0$ and such that \linebreak $\lra{\dr_{x_{i_1}}\wedge\dots\wedge \dr_{x_{i_k}}, \delta_2,\ldots,\delta_{\binom{n}{k}}}$ is a basis of $\bigwedge^k \Der$.
	\label{remar:supp:one:space}
	
\end{remar}
\begin{lem}\label{linalglemma}
	Let $R$ be a graded ring and $F$ be a free graded $R$-module of rank $n \in \NN_{>0}$ with bases $\mathcal B=\{b_1,\ldots,b_n\}$ and $\mathcal C=\{c_1,\ldots,c_n\}.$ For $k\in \{1, \ldots, n-1\},$ let $I, I_1, \ldots, I_k \subseteq R$ be homogeneous ideals. 
	Define the graded modules $V=\bigoplus_{i=1}^kI_i b_i \oplus \bigoplus_{j=k+1}^n Rb_j$ and $W = Ic_1 \oplus \bigoplus_{i=2}^n R c_i.$ If $\dim_\C (V_0\cap W_0)=n-k-1,$ then there exists a basis $\mathcal B'=\{b_1', \ldots, b_n'\}$ of $F,$ such that:
	\[V \cap W= \bigoplus_{i=1}^kI_i b_i' \oplus I b_{k+1}' \oplus \bigoplus_{j=k+2}^n Ab_j'.\]
\end{lem}
\begin{proof}
	Let $V'=\langle V_0 \rangle$ and $W'=\langle W_0 \rangle.$
	After renumbering the $b_i$ with index $i\geqslant k+1,$ we can assume $b_{k+1} \notin V_0\cap W_0.$ Then $\overline{\mathcal{B}} = \{\overline b_{k+1}\}$ is a basis of $F/W',$ which yields the existence of $a_i \in R$ and $w_i \in W',$ such that $b_i=a_ib_{k+1}+w_i$ for $i \in \{1, \ldots, k, k+2, \ldots, n\}$ and the existence of a unit $a_{k+1} \in R$ and of $w_{k+1} \in W'$ with $c_1=a_{k+1}b_{k+1}+w_{k+1}.$ This implies that $\mathcal B'=\{w_1,\ldots,w_k,b_{k+1},w_{k+2},\ldots,w_n\}$ is a basis of $F.$ We obtain 
	\[V=\bigoplus_{i=1}^kI_i w_i \oplus Rb_{k+1} \oplus \bigoplus_{j=k+2}^n Rw_j\] and \[W = \bigoplus_{i=1}^kR w_i \oplus Ib_{k+1} \oplus \bigoplus_{j=k+2}^n Rw_j.\]
	Then \[V\cap W = \bigoplus_{i=1}^kI_i w_i \oplus I b_{k+1} \oplus \bigoplus_{j=k+2}^n Rw_j.\]
\end{proof}


\begin{proof}[Proof of Theorem~\ref{theo:gen:free}]
	Let us prove Theorem~\ref{theo:gen:free} by induction. The initialization for $s=1$ is given by Lemma~\ref{lem:one:space}. Let $N=\binom{n}{k}$ and $s\in\lra{1,\ldots,N-1}$.
	
	We assume that $X_1,\ldots,X_{s+1}$ are linear subspaces of $\C^n$ of codimension $k$ which are in generic position.
	
	Let $X=\bigcup_{i=1}^s X_i$, $V=\Derlog{X}{k}, W=\Derlog{X_{s+1}}{k}$ and $F=S^N.$ By the induction hypothesis, $\dim_\C V_0=N-s$ and by Lemma~\ref{lem:one:space}, $\dim_\C W_0=N-1$. 
	Then $\dim_\C V_0\cap W_0=N-s-1$ follows from the genericity of the subspace arrangement. 
	By Proposition \ref{prop:inter} it holds that\[\Derlog{\left(\bigcup_{i=1}^{s+1}X_i\right)}{k}=V\cap W.\]  Then Lemma \ref{linalglemma} yields the result. 
\end{proof}

\begin{proof}[Proof of Corollary~\ref{cor:free}]
Let $\lra{\delta_1,\ldots,\delta_{N}}$ be a basis of $\bigwedge^k\Der$ such that a minimal generating set of $\Derlog{X}{k}$ is given by Equation~\eqref{eq:theo}. Since for all $i\in\lra{1,\ldots,s}$, $(h_{i,1},\ldots,h_{i,k})$ is a regular sequence, a minimal free resolution of the ideal $\lrb{h_{i,1},\ldots,h_{i,k}}$ is given by the truncated Koszul complex $\wt{K}_i:=\wt{K}(h_{i,1},\ldots,h_{i,k})$. Since 
\[
\Derlog{X}{k}=\bigoplus_{i=1}^s \lrb{h_{i,1},\ldots,h_{i,k}}\delta_i \oplus \bigoplus_{i=s+1}^N S \delta_i,
\]
we deduce that a minimal free resolution of $\Derlog{X}{k}$ is
\[
\wt{K}_1\oplus\dots\oplus\wt{K}_s\oplus \bigoplus_{i=s+1}^N \mc{C}
\]
where $\mc{C}$ is defined as in Notation~\ref{nota:kos}. 
Thus, the projective dimension of $\Derlog{X}{k}$ is $k-1$ and $X$ is free. 
\end{proof}

The following example shows that there exist subspace arrangements that are not, up to linear change of coordinates, unions of coordinate subspaces and that the genericity assumption cannot be dropped in Theorem~\ref{theo:gen:free}.
\begin{ex}\ 
	\begin{enumerate}
		\item Let us consider the generic subspace arrangement \[X=V(x,y)\cup V(z,t) \cup V(x-z,y-t) \in \C^4.\]
		We notice that the intersection of two of the components is always $0$-dimensional. If, up to a linear change of coordinates, the subspace arrangement would be a union of coordinate subspaces, this could not occur.
		Using this approach one can construct further examples of generic subspace arrangements, which are not union of coordinate subspaces in arbitrary dimensions.
		\item Let us consider the subspace arrangement $Y$ defined by the equations $h_1=xy(x-y+z-t)$ and $h_2=zt$. It is the union of $6$ planes in $\C^4$. Computations using \textsc{Singular} show that $Y$ is not free, since a minimal free resolution is given by:
		\[
		0\to S\to S^5\to S^{10}\to \Derlog{Y}{2}\to 0.
		\]
	\end{enumerate}
\end{ex}

Equation~\eqref{eq:der0} in the proof of Lemma \ref{lem:one:space} gives a correspondence between subspaces of codimension $k$ in $\C^n$ and some hyperplanes in $\C^{\binom{n}{k}}.$ Using this correspondence we can associate a hyperplane arrangement to any subspace arrangement. 
In the following example we investigate if there is a relation between the freeness of a subspace arrangement and the freeness of its associated hyperplane arrangement.

\begin{ex} \ 
	\begin{enumerate}
		\item We consider the subspace arrangement 
		\[X=V(x,y-z)\cup V(y,x+z)\cup V(x,y-t)\cup V(y,x+t)\cup V(x-y,z)\cup V(x,z-t)\cup V(x+t,z) \subseteq \C^4.\]
		By Equation~\eqref{eq:der0} we associate the hyperplane arrangement
		\[Y=V(x_1-x_2)\cup V(x_1-x_3)\cup V(x_1-x_4)\cup V(x_1-x_5)\cup V(x_2-x_3)\cup V(x_2-x_4)\cup V(x_2-x_6)\subseteq \C^6.\]
		Using \textsc{Singular} we can show that both $X$ and $Y$ are free.
		
		\smallskip
		
		One can show that for example the hyperplane $V(x_1-x_6)$ cannot be associated to a subspace of codimension $2$ in $\C^4,$ hence not all hyperplane arrangements in $\C^6$ can arise from subspace arrangements in this way.
		\item We consider the subspace arrangement 
		\[X= V(x,y)\cup V(x,z)\cup V(y,z)\cup V(x-z,y+z) \subseteq \C^3.\]
		By Equation~\eqref{eq:der0} we associate the hyperplane arrangement
		\[Y=V(x)\cup V(y)\cup V(z)\cup V(x+y+z) \subseteq \C^3.\]
		Since $\dim(X)=1,$ we obtain that $X$ is free, but a \textsc{Singular} computation shows that $Y$ is not free.
	\end{enumerate}
\end{ex}

\begin{remar}
The condition on the number of subspaces in Theorem~\ref{theo:gen:free} cannot be dropped, as we observed by considering randomly generated examples with more than $\binom{n}{k}$ subspaces with \textsc{Singular}.
\end{remar}

\section{Constructing free singularities via products}
\label{construction}
In this section we describe two ways of constructing new free singularities from known free singularities via two kinds of products: scheme-theoretic products and a generalization of the product in the sense of hyperplane arrangements.

\begin{nota}
Let $S_1=\CC\lra{x_1,\ldots,x_{n_1}}$ and $S_2= \CC\lra{y_1,\ldots,y_{n_2}}$. For the sake of simplicity, a germ of analytic space $(X,0)$  will be denoted by $X$. 

We set $S=S_1\hat{\otimes} S_2\simeq \C\lra{x_1,\ldots,x_{n_1},y_1,\ldots,y_{n_2}}.$
\end{nota}

\begin{nota}
\label{nota:X:Y}
The following notations are fixed in this section.

For $i\in\lra{1,2}$ let $X_i\subseteq \C^{n_i}$ be a reduced Cohen-Macaulay subspace of codimension $k_i$ and $(f_{i,1},\ldots,f_{i,k_i})\subseteq S_i$ be the equations of a reduced complete intersection $C_i$ of codimension $k_i$ containing $X_i$.  
\end{nota}

The next lemma recalls basic properties of analytic tensor products which will be used after.

\begin{lem}[\protect{\cite[Kapitel III \S 5 Satz 10, Satz 17, Satz 19]{GR71}}]
\label{lem:tens}
Let $R_1$ and $R_2$ be two analytic $\C$-algebras and $R=R_1\hat{\otimes} R_2$. 
Let $M_i$ be an $R_i$-module for $i\in\lra{1,2}$. Then 
\begin{enumerate}
\item $\depth_R (M_1\otimes M_2)=\depth_{R_1}(M_1)+\depth_{R_2}(M_2)$,
\item $\dim_R (M_1\otimes M_2)=\dim_{R_1}(M_1)+\dim_{R_2}(M_2)$.
\item $R_1$ and $R_2$ are reduced if and only if $R$ is reduced.
\end{enumerate}
\end{lem}

It follows that:
\begin{cor}
\label{cor:CM:prod}
With the hypothesis of Notation~\ref{nota:X:Y}, the product $X_1\times X_2\subseteq \C^{n_1}\times \C^{n_2}$ is a reduced Cohen-Macaulay subspace. 
\end{cor}

\begin{remar}
	Let $X \subseteq \C^n$ be a reduced Cohen-Macaulay subspace. 
	The freeness of $X$ is independent of the embedding in the following sense: \newline 
	Let $p \in \NN.$ If $X$ is free, then Lemma~\ref{lem:tens} and Proposition~\ref{prop:free:sing:loc} implies that $Y=X \times {(0,\ldots,0)} \subseteq \C^n\times \C^p$ is free. 
\end{remar}

\begin{nota}
We define $X:=X_1\times X_2$.
A reduced complete intersection $C$ containing $X$ is defined by the regular sequence $(f_{1,1},\ldots,f_{1,k_1},f_{2,1},\ldots,f_{2,k_2})\subseteq S$. In particular, $\mathrm{codim}(X)=\mathrm{codim}(C)=k_1+k_2$ and $J_{C}=SJ_{C_1}\cdot SJ_{C_2}$.
\end{nota}

The main result of this section is:

\begin{theo}
\label{theo:prod}
Let $X_1\subseteq \C^{n_1}$ and $X_2\subseteq\C^{n_2}$ be reduced Cohen-Macaulay subspaces and $X=X_1\times X_2\subseteq \C^{n_1}\times\C^{n_2}$. Then $X_1$ and $X_2$ are free if and only if $X$ is free.
\end{theo}

\begin{remar}
In particular, if $X_1$ and $X_2$ are hypersurfaces, then $X_1$ and $X_2$ are free divisors if and only if $X_1\times X_2$ is a free complete intersection of codimension~$2$. 
\end{remar}

We will need the following results.

\begin{lem}[\protect{\cite[Lemma 6.5.18]{dejong}}]
\label{lem:depth}
Let $R$ be a local Noetherian ring and consider a short exact sequence of $R$-modules : 
$$0\to M_1\to M_2\to M_3\to 0.$$
Then $$\depth(M_2) \geqslant \min\lrp{\depth(M_1),\depth(M_3)}.$$
In case this inequality is strict, we have $\depth(M_1)=\depth(M_3)+1$.
\end{lem}

\begin{lem}
\label{lem:summands:inter}
Let $R_1$ and $R_2$ be two analytic $\C$-algebras and $R=R_1\hat{\otimes} R_2$. Let $I\subseteq R_1$ and $J\subseteq R_2$. We assume that $\depth\lrp{R_1/I}<\depth(R_1)$ and $\depth\lrp{R_2/J}<\depth(R_2)$. Then:
\begin{enumerate}
\item \label{I+J} $\depth\lrp{R/(RI+RJ)}=\depth\lrp{R_1/I}+\depth\lrp{R_2/J}$,
\item $\depth\lrp{R/(RI\cap RJ)}=\depth\lrp{R_1/I}+\depth\lrp{R_2/J}+1$.
\end{enumerate}
\end{lem}
\begin{proof}\ 
\begin{enumerate}
\item The statement follows from Lemma~\ref{lem:tens} noticing that $R/(RI+RJ)\simeq (R_1/I)\hat{\otimes} (R_2/J)$.

\item Let us consider the exact sequence 
\begin{equation}
\label{eq:inter}
0\to R/(RI\cap RJ)\to (R/RI)\oplus(R/RJ)\to R/(RI+RJ)\to 0.
\end{equation}
Applying Lemma~\ref{lem:tens} to $R/RI=(R_1/I)\hat{\otimes}R_2$ yields 
\[
\depth(R/RI)=\depth(R_1/I)+\depth(R_2).
\]
By assumption $\depth(R_2)>\depth(R_2/J),$ hence \eqref{I+J} and Lemma~\ref{lem:tens} imply
\[\depth(R/RI)>\depth(R/(RI+RJ)).\] Analogously we obtain 
\[\depth(R/RJ)>\depth(R/(RI+RJ)).\]
Since $\depth((R/RI)\oplus (R/RJ))=\min(\depth(R/RI),\depth(R/RJ))$, we get 
\[
\depth((R/RI)\oplus (R/RJ))>\depth(R/(RI+RJ)).
\]
In this case the inequality in Lemma~\ref{lem:depth} is strict, hence
\[
\depth(R/(RI\cap RJ))=\depth(R/(RI+RJ))+1.
\] 
\end{enumerate}
\end{proof}

\begin{prop}
\label{prop:CM}
Let $R_1$ and $R_2$ be two analytic $\C$-algebras and $R=R_1\hat{\otimes} R_2$. Let $I\subseteq R_1$ and $J\subseteq R_2$. We assume that $\depth\lrp{R_1/I}<\depth(R_1)$ and $\depth\lrp{R_2/J}<\depth(R_2)$. Then the following are equivalent:
\begin{enumerate}
\item $R/(RI\cap RJ)$ is Cohen-Macaulay, 
\item \label{prop:CM:2} $R_1, R_2, R_1/I$ and $R_2/J$ are Cohen-Macaulay, $\dim(R_1/I)=\dim(R_1)-1$ and $\dim(R_2/J)=\dim(R_2)-1$.
\end{enumerate}
\end{prop}
\begin{proof}
	By Lemma~\ref{lem:summands:inter}, we have:
	\begin{equation}
		\label{eq:depth:IcapJ}
		\depth\lrp{R/(RI\cap RJ)}=\depth\lrp{R_1/I}+\depth\lrp{R_2/J}+1.
	\end{equation}
	
	Furthermore, Lemma~\ref{lem:tens}  and our assumptions imply the following inequality:
	\begin{eqnarray} \label{eq:dimineq}
		\dim(R/(RI\cap RJ))=&\max\left(\dim(R/RI), \dim(R/RJ)\right)\nonumber  \\
		=& \max\left(\dim(R_1/I)+\dim(R_2), \dim(R_1)+\dim(R_2/J) \right) \nonumber \\
		\geqslant& \min\left(\dim(R_1/I)+\dim(R_2), \dim(R_1)+\dim(R_2/J) \right) \nonumber \\
		\geqslant& \min\left(\depth(R_1/I)+\depth(R_2), \depth(R_1)+\depth(R_2/J) \right) \nonumber \\
		\geqslant  & \depth(R_1/I)+\depth(R_2/J)+1. 
	\end{eqnarray}
	
	Assume first that the hypothesis of the second statement is satisfied. In this case Inequality \eqref{eq:dimineq} becomes an equality.
	Then the first statement follows by using Equation~\eqref{eq:depth:IcapJ}. \newline 
	Next we assume that $R/(RI\cap RJ)$ is Cohen-Macaulay. Due to Equation~\eqref{eq:depth:IcapJ} and Inequality \eqref{eq:dimineq} we obtain: 
	\begin{eqnarray*}
		\depth\lrp{R/(RI\cap RJ)} = & \depth\lrp{R_1/I}+\depth\lrp{R_2/J}+1 \\
		\leqslant& \dim(R/(RI\cap RJ))
	\end{eqnarray*}
	Since $R/(RI\cap RJ)$ is Cohen-Macaulay, equality holds everywhere, which yields that $R_1, R_2, R_1/I$ and $R_2/J$ are Cohen-Macaulay and $\dim(R_2/J)=\dim(R_2)-1$ and $\dim(R_1/I)=\dim(R_1)-1$.
\end{proof}

\begin{lem}[\protect{\cite[Kapitel III, \S5 Korollar zu Satz 5]{GR71}}]
	\label{lem:IJ:IinterJ}
	Let $R_1$ and $R_2$ be two analytic $\C$-algebras 
	and $R=R_1\hat{\otimes} R_2$.
	Let $I \subseteq R_1$ and $J \subseteq R_2$ be ideals. Then the following equality holds in the ring $R$:
	\[
	RI\cdot RJ=RI\cap RJ.
	\]
\end{lem}

\begin{proof}[Proof of Theorem~\ref{theo:prod}]

We set for $i\in\lra{1,2}$, $R_i=S_i/I_{X_i}$ and \\ $R=S/I_X=S_1/I_{X_1}\hat{\otimes} S_2/I_{X_2}$. 

For $i\in\lra{1,2}$, let $J_{X_i/C_i}\subseteq S_i$ and $J_{X/C}\subseteq S$ be defined as in Notation~\ref{nota:JXC}. We denote by $\pi : S\to R,$ respectively $\pi_i: S_i \rightarrow R_i$ the canonical surjections. 
Then $J_{C}=SJ_{C_1}\cdot SJ_{C_2}\subseteq S,$ hence Lemma~\ref{lem:IJ:IinterJ} implies
\begin{eqnarray}
	\pi(J_{X/C})=& \pi(J_{C})\nonumber\\ 
	=& R\pi_1(J_{C_1}) \cdot R\pi_2(J_{C_2})\nonumber\\
	=&R\pi_1(J_{X_1/C_1})\cdot R\pi_2(J_{X_2/C_2}) \nonumber\\
	=&R\pi_1(J_{X_1/C_1})\cap R\pi_2(J_{X_2/C_2}) \label{eq:dd}
\end{eqnarray}

First we assume $J_{X_i/C_i} \neq S_i$ for $i \in \{1,2\}.$
Then, by Proposition~\ref{prop:free:sing:loc}, $X$ is free if and only if $R/\pi(J_{X/C})$ is Cohen-Macaulay of $R$-codimension $1.$ By Equation~\eqref{eq:dd} and Proposition~\ref{prop:CM} we obtain that  $R/\pi(J_{X/C})$ is Cohen-Macaulay if and only if for $i\in\lra{1,2}$ it holds that $R_i$ and $R_i/\pi_i(J_{X_i/C_i})$ are Cohen-Macaulay and $\dim(R_i)=\dim(R_i/\pi_i(J_{X_i/C_i}))+1.$ This is, again by Proposition~\ref{prop:free:sing:loc}, equivalent to the fact that $X_1$ and $X_2$ are free.
Next we consider the case $J_{X_i/C_i}=S_i$ for at least one $i \in \{1,2\}.$ 
In case $J_{X/C}=S$ the statement is obvious, hence we assume without loss of generality $J_{X/C}=SJ_{X_1/C_1}.$ Then $R/\pi(J_{X/C})\cong R_1/\pi_1(J_{X_1/C_1}) \hat{\otimes} R_2.$ In this setup the statement follows from Lemma \ref{lem:tens}.
\end{proof}

\begin{remar}
As a consequence, if $X_1$ and $X_2$ are free Cohen-Macaulay subspaces, we have 
\begin{multline*}
\projdim{\Derlog{X_1\times X_2}{k_1+k_2}}=\\
\projdim{\Derlog{X_1}{k_1}}+\projdim{\Derlog{X_2}{k_2}}+1
\end{multline*}
\end{remar}

A different notion of product for hyperplane arrangements is considered in \cite[Definition 2.13]{orlik-terao-hyperplanes}. It can be generalized to subspaces of higher codimension as follows:

\begin{de}
Let $X_1\subseteq \C^{n_1}$ and $X_2\subseteq \C^{n_2}$ be two equidimensional subspaces, both of the same codimension $k$. We set $X_1*X_2=X_1\times \C^m\cup \C^n\times X_2$.
\end{de}

\begin{nota}
\label{nota:prod}
Let $X_1\subseteq \C^{n_1}$ and $X_2\subseteq \C^{n_2}$ be two reduced equidimensional subspaces, both of the same codimension $k$. Let $X_1'=X_1\times \C^{n_2}$ and $X_2'=\C^{n_1}\times X_2$. 

For $i\in\lra{1,2}$ let $\iota_i : \bigwedge^k\mathrm{Der}_{\C^{n_i}}\to \bigwedge^k\mathrm{Der}_{\C^{n_1+n_2}}$ be the canonical maps.  We identify $\Derlog{X_i}{k}$ with the submodule of $\bigwedge^k\mathrm{Der}_{\C^{n_1+n_2}}$ generated by $\iota_i\lrp{\Derlog{X_i}{k}}$. 

Consider the decomposition:
\[ 
\bigwedge^k\mathrm{Der}_{\C^{n_1+n_2}}=D_1\oplus D_2\oplus D_{1,2}
\]
where $D_i$ is the submodule generated by the image of $\bigwedge^k \mathrm{Der}_{\C^{n_i}}$ in $\bigwedge^k\mathrm{Der}_{\C^{n_1+n_2}}$ and $D_{1,2}$  is the free submodule of $\bigwedge^k \mathrm{Der}_{\C^{n+m}}$ generated by the elements of the form $\dr_{x_{i_1}}\wedge \dots\wedge \dr_{x_{i_p}}\wedge \dr_{y_{j_1}}\wedge \dots\wedge \dr_{y_{j_{k-p}}}$ where $p\in\lra{1,\ldots,k-1}$.

\end{nota}

A similar result as Theorem~\ref{theo:prod} is satisfied, which generalizes \cite[Proposition 4.28]{orlik-terao-hyperplanes}:
\begin{prop}
Let $X_1\subseteq \C^{n_1}$ and $X_2\subseteq \C^{n_2}$ be two reduced equidimensional subspaces, both of the same codimension $k$. Then, with Notation~\ref{nota:prod}:
\[
\Derlog{X_1*X_2}{k}=\Derlog{X_1}{k}\oplus \Derlog{X_2}{k}\oplus D_{1,2}.
\]

 In particular, $X_1*X_2$ is free if and only if both $X_1$ and $X_2$ are free.
\end{prop}
\begin{proof}

   We have:
\[
\Derlog{X_1'}{k}=\Derlog{X_1}{k}\oplus D_{2}\oplus D_{1,2},
\]
\[
\Derlog{X_2'}{k}=D_1 \oplus \Derlog{X_2}{k} \oplus D_{1,2}.
\]

By Proposition~\ref{prop:inter}, $\Derlog{X_1*X_2}{k}=\Derlog{X_1'}{k}\cap\Derlog{X_2'}{k}$. We thus have the decomposition:
\[
\Derlog{X_1*X_2}{k}=\Derlog{X_1}{k}\oplus \Derlog{X_2}{k}\oplus D_{1,2}
\]

A minimal free resolution of $\Derlog{X_1*X_2}{k}$ is thus given as the direct sum of minimal free resolutions of $\Derlog{X_1}{k}, \Derlog{X_2}{k}$ and $D_{1,2}$. Since $D_{1,2}$ is free, the projective dimension of $\Derlog{X_1*X_2}{k}$ is 
\[
\max\lra{\projdim{\Derlog{X_1}{k}},\projdim{\Derlog{X_2}{k}}}.
\] 

Since by \cite[Proposition 4.2]{polfreeci}, $\projdim{\Derlog{X_i}{k}}\geqslant k-1$, we have $\projdim{\Derlog{X_1*X_2}{k}}=k-1$ if and only if 
\[
\projdim{\Derlog{X_1}{k}}=\projdim{\Derlog{X_2}{k}}=k-1.
\] 
\end{proof}

\bibliographystyle{alpha}
\bibliography{bibli2}

\end{document}